\documentclass[-ejs]{imsart}

\RequirePackage[OT1]{fontenc}
\RequirePackage{amsthm, amsmath, amsfonts, amssymb,epsf, graphicx}
\RequirePackage[square]{natbib}
\RequirePackage[colorlinks,citecolor=blue,urlcolor=blue]{hyperref}
\RequirePackage{hypernat}

\makeatletter \@addtoreset{equation}{section} \makeatother

\newcommand{\FF}{{\mathbb F}}
\newcommand{\HH}{{\mathbb H}}

\newcommand{\RR}{{\mathbb  R}}

\newcommand{\UU}{{\mathbb U}}

\newcommand{\GG}{{\mathbb G}}

\allowdisplaybreaks[1]

\newtheorem{thm}{Theorem}

\newtheorem{cor}{Corollary}

\advance \topmargin by -\headheight
\advance \topmargin by -\headsep
\advance \topmargin .2in
\begin{document}

\begin{frontmatter}
\title{Finite sampling inequalities:  an application to two-sample Kolmogorov-Smirnov statistics}
\runtitle{Finite sampling inequalities}

\begin{aug}
\author{Evan Greene} \thanksref{t1}\ead[label=e1]{egreene@uw.edu}
\and
\author{Jon A. Wellner,} \thanksref{t2}\ead[label=e2]{jaw@stat.washington.edu}



\thankstext{t1}{Supported by NSF Grant DMS-1104832}
\thankstext{t2}{Supported in part by NSF Grant DMS-1104832 and NI-AID grant R01 AI029168} 
\runauthor{Greene and Wellner}



\address{Department of Statistics, Box 354322\\University of Washington\\Seattle, WA  98195-4322}
\printead{e1}
\address{Department of Statistics, Box 354322\\University of Washington\\Seattle, WA  98195-4322}
\printead{e2}
\end{aug}

\begin{abstract}
We review a finite-sampling exponential bound due to Serfling and discuss
related exponential bounds for the hypergeometric distribution.   
We then discuss how such bounds motivate some new results for two-sample 
empirical processes.  Our development complements recent 
results by 
\cite{MR2887482} 
concerning exponential bounds for two-sided
Kolmogorov - Smirnov statistics by giving corresponding results for one-sided 
statistics with emphasis on ``adjusted'' inequalities of the type proved originally 
by \cite{MR0083864} 
and by 
\cite{MR1062069} 
for one-sample versions of these statistics.  
\end{abstract}



\end{frontmatter}


\bigskip
\section{Introduction:  Serfling's finite sampling exponential bound} 

Suppose that $\{ c_1, \ldots , c_N\}$ is a finite population with each $c_i \in \RR$.  
For $n \le N$, let $Y_1, \ldots , Y_n $ be a sample drawn from $\{c_1, \ldots , c_N\}$ without replacement;
we can regard the finite population $\{ c_1, \ldots , c_N \}$ as an urn containing $N$ balls labeled
with the numbers $c_1, \ldots , c_N$. 
Some notation:  we let
\begin{eqnarray*}
&& \mu_N = N^{-1} \sum_{i=1}^N c_i \equiv \overline{c}_N ,  \qquad  \sigma_N^2 = N^{-1} \sum_{i=1}^N (c_i - \overline{c}_N)^2,\\
&& a_N \equiv \min_{1 \le i \le N} c_i, \qquad b_N \equiv \max_{1 \le i \le N} c_i, \\
&& f_n \equiv \frac{n-1}{N-1}, \qquad \mbox{and} \qquad  f_n^* \equiv \frac{n-1}{N} .
\end{eqnarray*}
It is well-known (see e.g. \cite{Rice-2007}, Theorem B, page 208)
that $\overline{Y}_n = n^{-1} \sum_{i=1}^n Y_i$ satisfies
$E (\overline{Y}_n) = \mu_N$ and
\begin{eqnarray}
Var( \overline{Y}_n ) =  \frac{\sigma_N^2}{n} \left (1 - \frac{n-1}{N-1} \right ) = \frac{\sigma_N^2}{n} (1- f_n ) .
\label{FiniteSamplingVarianceOfMean}
\end{eqnarray}
\cite{MR0420967}, Corollary 1.1, 
shows that for all $\lambda > 0$
\begin{eqnarray}
P( \sqrt{n} (\overline{Y}_n - \mu_N) \ge \lambda ) \le \exp \left ( - \frac{2  \lambda ^2 }{ (1-f_n^*) (b_N-a_N)^2} \right ) .
\label{SerflingsExpBound}
\end{eqnarray}
This inequality is an inequality of the type proved by \cite{MR0144363} for sampling with replacement 
and more generally for sums of independent bounded random variables.
Comparing (\ref{FiniteSamplingVarianceOfMean}) and (\ref{SerflingsExpBound}), 
it seems reasonable to ask whether
 the factor $f_n^{*}$ in (\ref{SerflingsExpBound}) 
 can be improved to $f_n \equiv (n-1)/(N-1)$?  Indeed Serfling ends his paper (on page 47) with the remark:
 ``(it is) also of interest to obtain (\ref{SerflingsExpBound}) with the usual sampling fraction instead of $f_n^*$''.
Note that when $n=N$, $\overline{Y}_n = \mu_N$, and hence the probability in (\ref{SerflingsExpBound})
is $0$ for all $\lambda>0$, and the conjectured improvement of Serfling's bound agrees with this while
Serfling's bound itself is positive when $n=N$.

Despite related results  due to \cite{MR0345259, MR0345260, MR0345261}, 
 it seems that a definitive answer to this question is not 
yet known.    

A special case of considerable importance is the case when the numbers on the balls in the urn
are all $1$'s and $0$'s:  suppose that $c_1  = \cdots = c_D =1$, while $c_{D+1} , \ldots , c_N = 0$.
Then $X\equiv n\overline{Y}_n = \sum_{i=1}^n Y_i$ is well-known to have a 
Hypergeometric$(n,D,N)$ distribution
given by 
\begin{eqnarray*}
P\left ( \sum_{i=1}^n Y_i  = k \right ) = \frac{{D \choose k} {N-D \choose n-k}}{{N \choose n}},  \ \ \ 
\max \{ 0 ,  D+n-N \} \le k \le  \min\{ n, D \} .
\end{eqnarray*} 
In this special case $\mu_N = D/N$, $\sigma_N^2  = \mu_N (1-\mu_N)$, while $b_N=1$ and $a_N=0$.  
Thus Serfling's inequality (\ref{SerflingsExpBound}) becomes
\begin{eqnarray*}
P( \sqrt{n} ( \overline{Y}_n - \mu_N ) \ge \lambda ) \le 
\exp \left ( - \frac{2 \lambda^2}{1- f_n^*} \right ) \ \ \ \mbox{for all} \ \ \lambda >0 ,
\end{eqnarray*}
and the conjectured improvement is
\begin{eqnarray*}
P( \sqrt{n} ( \overline{Y}_n - \mu_N ) \ge \lambda ) \le 
\exp \left ( - \frac{2 \lambda^2}{1- f_n} \right ) \ \ \ \mbox{for all} \ \ \lambda >0 .
\end{eqnarray*}
Despite related results due to 
\cite{MR534946}  
and \cite{MR2206293}  
 it seems that a bound of the form in the last display remains unknown. 

We should note that an exponential bound of the Bennett type for the 
tails of the hypergeometric distribution does follow from results of 
\cite{MR681461}  
and
\cite{MR1093412};  see also \cite{MR1429082}.   
\smallskip

\begin{thm} 
\label{thm:EhmThm}
(Ehm, 1991)  
If $1 \le n \le D \wedge (N-D)$, then $\sum_{i=1}^n Y_i \stackrel{d}{=} \sum_{i=1}^n X_i$ 
where $X_i \sim \mbox{Bernoulli} (\pi_i)$, with $\pi_i \in (0,1)$, are independent.
\end{thm}
\smallskip

It follows from Theorem~\ref{thm:EhmThm} that 
\begin{eqnarray*}
&& n (D/N) = E \left ( \sum_1^n Y_i \right ) = E  \left ( \sum_1^n X_i \right ) = \sum_{i=1}^n \pi_i, \\
&& n \frac{D}{N} \left (1 - \frac{D}{N} \right ) (1 - f_n ) = Var \left ( \sum_1^n Y_i \right ) 
   = Var\left ( \sum_1^n X_i\right ) =  \sum_{i=1}^n \pi_i (1-\pi_i) .
\end{eqnarray*}
Furthermore, by applying Theorem~\ref{thm:EhmThm} together with Bennett's inequality 
(\cite{Bennett:62}; see also \cite{MR838963}, page 851), we obtain the following 
exponential bound for the tail of the hypergeometric distribution:
\smallskip

\begin{cor} 
If $1 \le n \le D \wedge (N-D)$, then for all $\lambda>0$ 
\begin{eqnarray*}
P( \sqrt{n} (\overline{Y}_n - \mu_N ) \ge \lambda) 
\le \exp \left ( - \frac{ \lambda^2}{2 \sigma_N^2 (1-f_n)} \psi 
        \left ( \frac{\lambda}{\sqrt{n} \sigma_N^2 (1-f_n)} \right ) \right) 
\end{eqnarray*}
where $\mu_N \equiv D/N$, $\sigma_N^2 \equiv \mu_N (1-\mu_N)$,  $1-f_n \equiv 1- (n-1)/(N-1)$ 
is the finite-sampling correction factor, and $\psi (y) \equiv 2 y^{-2} h(1+y) $ where $h(y) \equiv y(\log y -1) +1$.
\end{cor}
\smallskip

Since $\sigma_N^2 = \mu_N (1- \mu_N) \le 1/4$, the inequality of the corollary yields 
a further bound which is quite close to the conjectured Hoeffding type improvement of 
Serfling's bound, and which now has the desired finite-sampling correction factor $1-f_n$:
\smallskip

\begin{cor}
\begin{eqnarray*}
P( \sqrt{n} ( \overline{Y}_n - \mu_N ) \ge \lambda) 
& \le & \exp \left (- \  \frac{2 \lambda^2}{(1-f_n)} \psi \left ( \frac{\lambda}{\sqrt{n} \sigma_N^2 (1-f_n)} \right ) \right ) \\
& \le & \exp \left (- \  \frac{2 \lambda^2}{(1-f_n)} \psi \left ( \frac{1}{\sigma_N^2 (1-f_n)} \right ) \right ) .
\end{eqnarray*}
\end{cor}

By considerations related to the work of 
\cite{MR1258865}  
and 
\cite{MR1973309},   
the first author of this paper has succeeded in proving the following 
exponential bound.
\smallskip

\begin{thm} 
\label{thm:Greene-W-HyperGeomBound} 
(Greene, 2014)  
Suppose that  $\sum_{i=1}^n Y_i \sim \mbox{Hypergeometric} (n, D, N)$.  
Define $\mu_N = D/N$ and suppose $N>4$ and $2 \le n < D \le N/2$.  Then for all $0 < \lambda < \sqrt{n}/2$ we have 
\begin{eqnarray*}
\lefteqn{P\left ( \sqrt{n} ( \overline{Y}_n - \mu_N) \ge \lambda \right )} \\ 
& \le & \sqrt{\frac{1}{2\pi \lambda^2}} \left ( \frac{1}{2} \right ) 
                \sqrt{ \left ( \frac{N-n}{N} \right ) \left ( \frac{\sqrt{n} + 2\lambda}{\sqrt{n} - 2 \lambda }\right ) 
                         \left ( \frac{N- n + 2 \sqrt{n} \lambda}{N-n - 2 \sqrt{n} \lambda }\right ) } \\
&& \ \ \cdot \exp \left ( - \frac{2}{1-\frac{n}{N}} \lambda^2 \right ) 
              \exp \left ( - \frac{1}{3} \left ( 1 + \frac{n^3}{(N-n)^3} \right ) \frac{\lambda^4}{n} \right ) .
\end{eqnarray*}
\end{thm}

The proof of this bound, along with a complete analogue for the 
hypergeometric distribution of a bound of Talagrand (1994) for the binomial distribution,
appears in \cite{Greene-Wellner:2015} and in the forthcoming Ph.D. thesis of the first author, 
\cite{Greene:2016}.

The bound given in Theorem~\ref{thm:Greene-W-HyperGeomBound}
involves a still better finite-sampling correction factor, namely
$1- \overline{f}_n = 1- n/N$, which  
has also appeared in \cite{MR856817} in the context of a Bayesian analysis of
finite sampling.  Note that as $N\rightarrow \infty$, the above bound yields 
\begin{eqnarray*}
\lefteqn{\limsup_{N\rightarrow \infty} P\left ( \sqrt{n} ( \overline{Y}_n - \mu_N) \ge \lambda \right )} \\ 
& \le & \sqrt{\frac{1}{2\pi \lambda^2}} \left ( \frac{1}{2} \right ) 
                \sqrt{  \left ( \frac{\sqrt{n} + 2\lambda}{\sqrt{n} - 2 \lambda }\right ) } 
\cdot \exp \left ( - 2 \lambda^2 - \frac{\lambda^4}{3n} \right )  ,
\end{eqnarray*}
a bound which improves slightly on the bound given by \cite{MR1973309}
in the case of sums of i.i.d. Bernoulli random variables.  

Before leaving this section we begin to make a connection to finite-sampling empirical 
distributions:  
Now let $\FF_n (t) = n^{-1} \sum_{i=1}^n 1_{(-\infty , t]} (Y_i ) $ and $F_N (t) = N^{-1} \sum_{i=1}^N 1_{(-\infty, t]} (c_i ) $.
Then it is easily seen that Serfling's bound yields
\begin{eqnarray*}
P( \sqrt{n} ( \FF_n (t) - F_N (t) ) \ge \lambda ) \le \exp \left ( - \frac{2 \lambda^2}{(1 - (n-1)/N)} \right ) 
\end{eqnarray*}
for each fixed $\lambda >0$ and $t \in \RR$.   Note that since $\FF_n (t)$ is equal in distribution to 
 the sample mean 
of $n$ draws without replacement from an urn containing $N F_N (t)$ $1$'s and $N(1-F_N (t))$ $0$'s, the 
bound in the last display only involves the hypergeometric special case of Serfling's inequality.
This leads to the following conjecture concerning bounds for the finite sampling empirical process 
$\{ \sqrt{n} (\FF_n (t) - F_N (t) ) : \ t \in \RR \}$:
\medskip

\par\noindent
{\bf Conjecture:}  There exist constants $C, D>0$ (possibly $C=1$ and $D=2$?) such that 
\begin{eqnarray}
&& P\left ( \sqrt{n} \sup_{t} ( \FF_n (t) - F_N (t) ) \ge \lambda \right ) 
           \le C \exp \left ( - \frac{2 \lambda^2}{(1 - f_n)} \right ) , 
           \label{ConjectureInequalKS-FSOneSided}\\
&& P\left ( \sqrt{n} \sup_{t} | \FF_n (t) - F_N (t) | \ge \lambda \right ) 
          \le D \exp \left ( - \frac{2 \lambda^2}{(1 - f_n)} \right ) 
           \label{ConjectureInequalKS-FSTwoSided}
\end{eqnarray}
for all $\lambda>0$.   The possibility that $D=2$ is suggested by 
the corresponding inequality established by 
\cite{MR1062069} 
in the case of sampling with replacement.  

With these strong indications of the plausibility of an improvement of Serfling's bound 
and corresponding improvements in exponential bounds for the uniform-norm deviations
of the finite-sampling empirical process,
we can now turn to an application of the basic idea 
in the context of two-sample Kolmogorov-Smirnov statistics.

\section{Two-sample tests and  finite-sampling connections}
\hfill \\
To connect this with the two-sample Kolmogorov-Smirnov statistics, suppose that 
$X_1, \ldots , X_m$ are i.i.d. $F$ and $Y_1, \ldots , Y_n$ are i.i.d. $G$.  Let $N=m+n$.  
Then for testing 
$H_c : F=G$ with $F$ continuous versus $K^+  : \ F \ge G$ ($F \prec_s G$), 
$K^{-} : \ G \ge F$, ($G \prec_s F$),  
or $K:  \ F \not= G$, the classical 
K-S test statistics are 
\begin{eqnarray*}
&& D_{m,n}^+ \equiv \sqrt{\frac{mn}{N}} \sup_x ( \FF_m (x) - \GG_n (x)) , \\
&& D_{m,n}^- \equiv \sqrt{\frac{mn}{N}} \sup_x ( \GG_n (x) - \FF_m (x)) , \ \ \mbox{and}\\
&& D_{m,n} \equiv \sqrt{\frac{mn}{N}} \sup_x | \FF_m (x) - \GG_n (x) | ,
\end{eqnarray*}
respectively.  It is well-known that under $H_c$ we have
\begin{eqnarray*}
D_{m,n}^{\pm}  \rightarrow_d \sup_{0 \le t \le 1} \UU (t) , \qquad 
D_{m,n} \rightarrow_d \sup_{0 \le t \le 1} | \UU (t) | 
\end{eqnarray*}
if $m \wedge n \rightarrow \infty$ where $\UU$ is a standard Brownian bridge process
on $[0,1]$; see e.g. 
\cite{MR0229351}, 
pages 189-190,  
\cite{MR0097136}, and 
\cite{MR1385671},  
pages 360-366.    

Note that with $\lambda_N \equiv m/N$ and 
$$
\HH_N \equiv \lambda_N \FF_m + (1-\lambda_N) \GG_n = N^{-1} \sum_{i=1}^N 1_{(-\infty, \cdot ]} (Z_{(i)} )
$$
where $Z_{(1)} \le \cdots \le Z_{(N)} $ are the order statistics of the pooled sample,
we have 
\begin{eqnarray*}
\FF_m - \HH_N 
& = & \FF_m - \lambda_N \FF_m - (1-\lambda_N) \GG_n 
= (1-\lambda_N) (\FF_m - \GG_n ), \qquad \mbox{and}\\
\GG_n - \HH_N 
& = & \GG_n - \lambda_N \FF_m - (1-\lambda_N) \GG_n = \lambda_N ( \GG_n - \FF_m ),
\end{eqnarray*}
and hence, with $\overline{\lambda}_N = 1- \lambda_N$,
\begin{eqnarray*}
&& \sqrt{\frac{mn}{N}} ( \FF_m - \GG_n ) 
 =  \sqrt{N} \sqrt{\lambda_N \overline{\lambda}_N} \frac{1}{\overline{\lambda}_N} (\FF_m - \HH_N) 
 =  \frac{1}{\sqrt{\overline{\lambda}_N}} \sqrt{m} ( \FF_m - \HH_N), \\
&& \sqrt{\frac{mn}{N}} ( \GG_n - \FF_m ) 
 =  \sqrt{N} \sqrt{\lambda_N \overline{\lambda}_N} \frac{1}{\lambda_N} (\GG_n - \HH_N) 
 =  \frac{1}{\sqrt{\lambda_N}} \sqrt{n} ( \GG_n - \HH_N) .
\end{eqnarray*}
Thus, using the independence of the ranks $\underline{R}$ and the order statistics $\underline{Z}$ 
\begin{eqnarray*}
P( D^+_{m,n} \ge t ) 
& = & E_Z P_R \left ( \sqrt{m} \| ( \FF_m - \HH_N )^{+} \|_{\infty} > t \sqrt{1-\lambda_N} \right ) 
\end{eqnarray*}
and it would follow from (\ref{ConjectureInequalKS-FSOneSided}) that 
\begin{eqnarray}
P( D_{m,n}^+ \ge t ) 
& \le  & C \exp \left ( - 2 \overline{\lambda}_N t^2 / (1 - f_m) \right ) \nonumber \\
& \le & C \exp \left ( - 2 (n/N) t^2 / (n/(N-1)) \right )  \nonumber \\
& = & C \exp \left ( - 2 \frac{N-1}{N} t^2 \right )   \label{TwoSampleKS-OneSidedProbBoundsConj}
\end{eqnarray}
for all $t>0$.  Similarly it would also follow from (\ref{ConjectureInequalKS-FSOneSided}) that 
\begin{eqnarray*}
P( D_{m,n}^- \ge t ) 
& \le  & C \exp \left ( - 2 \lambda_N t^2 / (1 - f_n ) \right )  \\
& \le & C \exp \left ( - 2 (m/N) t^2 / (m/(N-1)) \right ) = C \exp \left ( - 2 \frac{N-1}{N} t^2 \right )
\end{eqnarray*}
for all $t>0$.  
Combining the two one-sided inequalities yields a (conjectured) two-sided inequality:  
\begin{eqnarray*}
P( D_{m,n} \ge t ) 
& \equiv & P( \sqrt{mn/N} \| \FF_m - \GG_n \|_{\infty} > t ) \\
& \le & P( D_{m,n}^+ > t ) + P( D_{m,n}^{-} > t ) \\
& \le & 2C \exp \left ( - 2 \frac{N-1}{N} t^2 \right ) .
\end{eqnarray*}

In the next section we will prove that bounds of this type with $C=1$ and $D=2$ hold 
in the special case $m=n$.
For some results for the two-side two-sample Kolmogorov-Smirnov statistic in the case $m=n$ and
computational results for $m \not= n$, see 
\cite{MR2887482}.  These authors were aiming for a bound of the 
form $ C \exp (-2 t^2)$ both for $m=n$ and $m\not= n$.  
The above heuristics seem to suggest that a bound of the form
$C \exp ( - 2((N-1)/N) t^2)$ might be a natural goal.

\section{An exponential bound for $D_{m,n}^+$ when $m=n$}
\hfill \\
Throughout this section we suppose that the null hypothesis 
$H_c$ holds:  $G=F$ is a continuous distribution function.

From \cite{MR0097136}, (2.3) on page 473 (together with $t = \sqrt{mn/N} d$ and $d = a/n$ from page 473, line 4), 
when $m=n$ (so $N=2n$),
\begin{eqnarray*}
P( D_{n,n}^+ \ge t ) 
& = & P\left ( \sqrt{\frac{mn}{N}} \sup_x ( \FF_m (x) - \GG_n (x) ) \ge  \sqrt{\frac{mn}{N}} \frac{a}{n} \right ) \\
& = & P\left ( \sqrt{\frac{n^2}{2n}} \sup_x ( \FF_n (x) - \GG_n (x) ) \ge  \sqrt{\frac{n^2}{2n}} \frac{a}{n} \right ) \\
& = & \frac{{2n \choose n-a}}{{2n \choose n}}  \ \ \mbox{for} \ \ a = 1, 2, \ldots , n .
\end{eqnarray*}
We first compare the exact probability from the last display with the possible upper bounds
\begin{eqnarray*}
&& PB_2 (n) = \exp \left ( - 2 \frac{2n-1}{2n} \frac{a^2}{2n} \right ) ; \\
&& PB_3(n) = \exp \left ( - 2 \frac{a^2}{2n} \right ) .
\end{eqnarray*}
For $n=3$ we find that 
\begin{eqnarray*}
\begin{array}{| r | r  r  r   r  |} \hline
                                      a        &  0 & 1         &  2        & 3    \\ \hline 
                                   E(xact)   & 1 & .75       & 0.3      & 0.05 \\  \hline
                                    PB2    & 1 & 0.7574 & 0.3291 & 0.0821  \\
                                    PB2-E & 0 & 0.0074 & 0.0291 & 0.0321 \\ \hline
                                    PB3    & 1 &  0.7165 & 0.2636 & 0.0498 \\
                                    PB3-E  & 0 & -0.0335 & -0.0364 & - 0.0002 \\ \hline
\end{array}
\end{eqnarray*}
Further comparisons for $m=n = 10, 12, 13, 14, 15, 25$  support the validity 
of the bound involving the finite sampling fraction $f_n$.  
These comparisons agree with the following theorem:
\smallskip

\begin{thm} 
\label{BasicGW-TwoSampleKS-bounds}
A. \ When $m=n$ (so that $N=2n$) the second bound in 
(\ref{TwoSampleKS-OneSidedProbBoundsConj}) holds for all $n \ge 1 $  with $C=1$:
\begin{eqnarray}
P( D_{n,n}^+ \ge t) 
& = & P\left ( \sqrt{\frac{mn}{N}} \sup_x ( \FF_m (x) - \GG_n (x) ) \ge t \right ) \\
& \le & \exp \left ( - 2 \frac{N-1}{N} t^2 \right )  \ \ \ \mbox{for all} \ \ t>0 .               
\label{TwoSampleKS-OneSidedProbBoundsLittleMequalsLittleN}
\end{eqnarray}
Equivalently, when $m=n$,
\begin{eqnarray}
P\left ( \sqrt{\frac{mn}{N}} \sqrt{\frac{N-1}{N}} \sup_x ( \FF_m (x) - \GG_n (x) ) \ge t \right ) 
 \le  \exp \left ( - 2 t^2 \right )               
\label{TwoSampleKS-OneSidedProbBoundsLittleMequalsLittleNAlternateForm}
\end{eqnarray}
for all  $ t>0$.  In fact (via a much easier proof!) \\
\begin{eqnarray}
P ( D_{n,n}^+ \ge t ) \le  \exp \left ( - 2 \frac{N}{N+1} t^2 \right )  \ \ \ \mbox{for all} \ \ t>0 .
\label{TwoSampleKS-OneSidedProbBoundLittleMequalsLittleNLutzDuembgenForm}
\end{eqnarray}
B. \ On the other hand, when $m=n$ (so that $N=2n$), for all $n\ge 1$ we have
\begin{eqnarray*}
P( D_{n,n}^+ \ge t) > \exp ( - 2 t^2 ) \ \ \ \mbox{for all} \ \ 0 < t < 1.
\end{eqnarray*}
\end{thm}

\begin{proof}[\bf Proof]
A.  \ Since the inequality holds trivially for $a=0$, and can be shown easily by numerical 
computation for $a \in \{ 1, 2, 3 \}$ (see the Table above), 
it suffices to show that 
\begin{eqnarray*}
 \frac{{2n \choose n-a}}{{2n \choose n}}  \le \exp \left ( - 2 \frac{2n-1}{2n} \frac{a^2}{2n} \right ) 
\end{eqnarray*}
for $a \in \{ 1, \ldots , n \}$ and $n\ge 4$.  Furthermore, we will show that it holds for $a=n$ 
in a separate argument, and thus it suffices to show that it holds for $a \in \{ 1, \ldots , n-1 \}$ and $n\ge 4$.
By rewriting the numerator and denominator on the left side of the last display, the desired inequality 
can be rewritten as 
\begin{eqnarray*}
\frac{n! n!}{(n-a)! (n+a)!} \le \exp \left ( - \frac{2n-1}{2n} \cdot \frac{a^2}{n} \right ) .
\end{eqnarray*}
By taking logarithms we can rewrite this as 
\begin{eqnarray}
\log \left ( \frac{n! n!}{(n-a)!(n+a)!} \right ) + \frac{2n-1}{2n} \frac{a^2}{n} \le 0 .
\label{BasicLogExpression}
\end{eqnarray}
Now by Stirling's formula with bounds (see e.g. \cite{MR0117359})  
we have
\begin{eqnarray}
\phantom{blabl}\sqrt{2\pi k} \left (\frac{k}{e} \right )^k \exp \left ( \frac{1}{12k} - \frac{1}{360 k^3} \right ) 
 \le k! \le \sqrt{2\pi k} \left (\frac{k}{e} \right )^k \exp \left ( \frac{1}{12k}  \right ) .
 \label{StirlingsFormula}
 \end{eqnarray}
Using these bounds in (\ref{BasicLogExpression}) we find that the left side is 
bounded above by 
\begin{eqnarray*}
\lefteqn{ 
- n \left \{ (1- \frac{a}{n} ) \log (1- \frac{a}{n} ) + (1 + \frac{a}{n} ) \log (1 + \frac{a}{n} ) \right \} } \\
&& \ \ - \  \frac{1}{2} \left ( \log \left (1 - \frac{a}{n} \right ) + \log \left (1 + \frac{a}{n} \right ) \right ) \\
&& \ \ + \  \left \{ \frac{1}{6n} - \frac{1}{12 (n-a)} - \frac{1}{12(n+a)}
                + \frac{1}{360} \left ( \frac{1}{(n-a)^3} + \frac{1}{(n+a)^3} \right ) \right \} \\
&& \ \ + \ \frac{a^2}{n} - \frac{a^2}{2 n^2} \\
& \equiv & I_1 + I_2 + I_3 + I_4 .
\end{eqnarray*}
Note that $I_1$ and $I_2$ are as defined in \cite{MR2887482} page 640, while $I_3$ and $I_4 $ differ.
From \cite{MR2887482} page 640, 
\begin{eqnarray}
I_1 \le  - \frac{a^2}{n} - \frac{a^4}{6n^3} - \frac{a^6}{15 n^5} - \frac{a^8}{28 n^7} ,
\label{IOneBound}
\end{eqnarray}
(which is proved by Taylor expansion of $(1+x) \log (1+x) + (1-x) \log (1-x)$ about $x=0$), and 
\begin{eqnarray}
I_2 \le \frac{a^2}{2n^2} + \frac{a^4}{4 n^4} + \frac{a^6}{6 n^6 (1 - a^2/n^2)} .
\label{ITwoBound}
\end{eqnarray}
Note that the lead term in the bound (\ref{IOneBound}) for $I_1$ and lead term of $I_4$ cancel each other,
while the first term of the bound (\ref{ITwoBound}) for $I_2$ cancels the second term of $I_4$. 
Adding the bounds yields 
\begin{eqnarray*}
\lefteqn{I_1 + I_2 + I_3 + I_4 }\\
& \le & - \frac{a^4}{12 n^3}- \frac{a^4}{12 n^3} - \frac{a^6}{15 n^5} - \frac{a^8}{28 n^7} \\
&& \ \ + \ \frac{a^4}{4 n^4} + \frac{a^6}{6n^6 (1- a^2/n^2)} + I_3 \\
& = & - \frac{a^4}{n^3} \left ( \frac{1}{12} - \frac{1}{4n} \right )- \frac{a^4}{12 n^3}
            - \frac{a^6}{n^5} \left ( \frac{1}{15} - \frac{1}{6n (1 - a^2/n^2)} \right ) - \frac{a^8}{28 n^7} 
             \  + \ I_3 \\ 
& \le & - \frac{a^4}{n^3} \left ( \frac{1}{12} - \frac{1}{4n} \right ) - \frac{a^4}{12 n^3}
             - \frac{a^6}{n^5} \left ( \frac{1}{15} - \frac{1}{6 ( 2 - 1/n)} \right ) - \frac{a^8}{28 n^7}  + I_3\\
& \le & - \frac{a^4}{n^3} \left ( \frac{1}{12} - \frac{1}{4n} \right ) - \frac{a^4}{12 n^3}
             + \frac{3 a^6}{105 n^5}  - \frac{a^8}{28 n^7}  + I_3\\
& = & - \frac{a^4}{n^3} \left ( \frac{1}{12} - \frac{1}{4n} \right ) 
          - \frac{a^4}{12 n^3}\left(1-  \frac{36 a^2}{105 n^2}\right)  - \frac{a^8}{28 n^7}  + I_3\\
& \leq & - \frac{a^4}{n^3} \left ( \frac{1}{12} - \frac{1}{4n} \right ) - \frac{a^4}{21 n^3}  - \frac{a^8}{28 n^7}  + I_3\\
& \equiv & R_{12} + I_3 .
\end{eqnarray*}
Now $R_{12} \le 0 $ for $n\ge 4$ and $I_3 \le 0$ for all $n\ge 2$ and $a \in \{ 1, \ldots , n-1\}$
by the following argument:
\begin{eqnarray*}
I_3 & = & \frac{1}{6n} - \frac{1}{12 (n-a)} - \frac{1}{12(n+a)} + \frac{1}{360} \left ( \frac{1}{(n+a)^3} + \frac{1}{(n-a)^3} \right ) \\
& = & - \frac{1}{6}  \frac{a^2}{n (n^2 - a^2)}  + \frac{2}{360} \frac{n (n^2 + 3a^2)}{ (n^2 - a^2)^3} \\
& = & - \frac{1}{6 n (n^2-a^2)} \left \{  a^2  -  \frac{2}{60} \frac{n^2 (n^2+3a^2)}{(n^2-a^2)^2}  \right \} \\
& = & - \frac{1}{6 n (n^2 - a^2)} \left \{ a^2 - \frac{1}{30} \frac{n^2 (n^2 - a^2 + 4a^2)}{(n^2 -a ^2)^2} \right \} \\
& = & -\frac{1}{6 n (n^2 - a^2)} \left \{  a^2 \left ( 1 - \frac{2}{15} \frac{n^2}{(n^2 - a^2)^2}  \right ) 
              - \frac{n^2 (n^2-a^2)}{30 (n^2 - a^2)^2} \right \} \\
& \le & -\frac{1}{6 n (n^2 - a^2)} \left \{  a^2 \left ( 1 - \frac{2}{15} \frac{1}{3}  \right ) 
              - \frac{n^2 }{30 (n^2 - a^2)} \right \} \\
&& \ \ \ \mbox{by using  } \ a \le n-1, \ \mbox{so} \ \ n^2 - a^2 \ge n^2 - (n-1)^2 = (2n-1) , \\
&& \ \ \ \mbox{and } \ n^2 /(2n-1)^2 \le 1/3 \ \ \mbox{for} \ \ n \ge 4, \\
& = & -\frac{1}{6 n (n^2 - a^2)} \left \{  a^2 \left ( 1 - \frac{2}{3\cdot 15}   \right ) 
              - \frac{n^2- a^2 + a^2  }{30 (n^2 - a^2)} \right \} \\
& = &  -\frac{1}{6 n (n^2 - a^2)} \left \{  a^2 \left ( 1 - \frac{2}{3\cdot 15}  - \frac{1}{30 (n^2-a^2)}   \right ) 
              - \frac{1}{30} \right \} \\
& \le & -\frac{1}{6 n (n^2 - a^2)} \left \{  a^2 \left ( 1 - \frac{2}{3\cdot 15}  - \frac{1}{30 (2n-1)}   \right ) 
              - \frac{1}{30} \right \} \\
& \le & -\frac{1}{6 n (n^2 - a^2)} \left \{  a^2 \left ( 1 - \frac{31}{630}   \right ) 
              - \frac{1}{30} \right \} \\
\end{eqnarray*}
for $n\ge 4$.  This is a decreasing function of $a$ for fixed $n$, and hence to show that it is $<0$ it suffices to 
check it for $a=1$.  But when $a=1$ the right side above equals 
\begin{eqnarray*}
\lefteqn{- \frac{1}{6 n(n^2-1^2)}  \left \{ 1 - \frac{31}{630} - \frac{1}{30}  \right \} }\\
 &  = &  - \frac{1}{ n(n^2-1)}  \left \{ \frac{289}{6\cdot 315} \right \}  
  <  - \frac{1}{n(n^2-1)}   \left \{ \frac{280}{6\cdot 315} \right \} 
 =  - \frac{4}{27 n(n^2-1)} <0,
\end{eqnarray*}
so we conclude that $I_3 <0$ for $a \in \{ 1, \ldots , n-1 \}$ and $n\ge 4$.  
It remains only to show that the desired bound holds for $a=n$; that is we have
\begin{eqnarray*}
\frac{1}{{2n\choose n}} \le \exp (- (n-1/2) ) .
\end{eqnarray*}
But this  can easily be shown via the Stirling formula bounds (\ref{StirlingsFormula}).

Thus 
\begin{eqnarray*}
\exp ( I_1 + I_2 + I_3 ) \le \exp ( - I_4 ) = \exp \left ( - \frac{2n-1}{2n} \frac{a^2}{n} \right ),
\end{eqnarray*}
and the claimed inequality holds for all $n\ge 4$. Since the bounds hold for $n=1,2,3$ by 
direct numerical computation, the claim follows.

Here is the simple proof of (\ref{TwoSampleKS-OneSidedProbBoundLittleMequalsLittleNLutzDuembgenForm})
due to Lutz D\"umbgen:  For integers $a \in \{ 1, \ldots , n \}$
\begin{eqnarray*}
\log \left ( {2n \choose n+a } \big /  {2n \choose n}  \right )
& = & \log \frac{(n!)^2}{(n-a)!(n+a)!} = \log \prod_{i=1}^a \frac{n+1-i}{n+i} \\
& = & \sum_{i=1}^a \log \frac{n+1-i}{n+i} = \sum_{i=1}^a \log \frac{n+1/2 -(i-1/2)}{n+1/2 +(i-1/2)} \\
& = & \sum_{i=1}^a \log \frac{1 - (i-1/2)/(n+1/2)}{1 + (i-1/2)/(n+1/2)} < - 2 \sum_{i=1}^a \frac{i-1/2}{n+1/2} \\
& = & - 2 \frac{a (a+1)/2 - a/2}{n+1/2} = - \frac{a^2}{n+1/2} = -2 \frac{N}{N+1} \frac{a^2}{2n} \\
& = & - 2 \frac{N}{N+1} t^2 .
\end{eqnarray*}
Here the inequality follows from 
\begin{eqnarray*}
\log \frac{1-x}{1+x} = - 2 \sum_{k=0}^\infty \frac{x^{2k+1}}{2k+1} < -2x \ \ \mbox{for} \ \ 0 < x < 1 .
\end{eqnarray*}
Note that $N/(N+1) \ge (N-1)/N$.  
\medskip 

B. \  We first define
\begin{eqnarray*}
r_n (a) & \equiv & \log \left \{ \frac{ {2n \choose n-a} / {2n \choose n} }{ \exp (- 2 a^2/(2n))}  \right \} \\
& = & \log {2n \choose n-a} - \log {2n \choose n}  + \frac{a^2}{n} .
\end{eqnarray*}
Since we can take $t = a/\sqrt{2n}$, 
it suffices to show that $r_n (a) >0$ for $1 \le a \le \lfloor \sqrt{2n} \rfloor$.
We will first show this for $n \ge 31$.  Then the proof will be completed by checking the inequality 
numerically for $1 \le a \le \lfloor \sqrt{2n} \rfloor$ and $ n \in \{ 1, \ldots , 30 \}$.

By using the Stirling formula bounds of (\ref{StirlingsFormula})
as in the proof of A, but now with upper bounds replaced by lower bounds, we find that
\begin{eqnarray*}
r_n (a) & = & 2 \log (n!) - \log (n-a)! - \log (n+a)! + \frac{a^2}{n} \\
& \ge & - n \left \{ \left (1 - \frac{a}{n} \right ) \log  \left (1 - \frac{a}{n} \right )  
           +  \left (1 + \frac{a}{n} \right ) \log  \left (1 + \frac{a}{n} \right ) \right \} \\
&& \ \ \ - \ \frac{1}{2} \left \{ \log  \left (1 - \frac{a}{n} \right ) + \log  \left (1 + \frac{a}{n} \right ) \right \} \\
&& \ \ \ + \ \frac{1}{6n} - \frac{1}{180 n^3} - \frac{1}{12 (n-a)} - \frac{1}{12 (n+a)} \\
&& \ \ \ + \ \frac{a^2}{n} \\
& \equiv & L_1 + L_2 + L_3 + L_4 .
\end{eqnarray*}
As in (\ref{IOneBound}) and (\ref{ITwoBound}) and the displays following them
we find that 
\begin{eqnarray*}
L_1 & \ge & - n \left \{ \frac{a^2}{n^2} + \frac{a^4}{6 n^4} + \frac{a^6}{15 n^6} 
                    + \frac{a^8}{28n^8} \left ( \frac{n^2}{ n^2 - a^2} \right )\right \}, \\
L_2 & \ge & \frac{a^2}{2n^2} + \frac{a^4}{4 n^4} + \frac{a^6}{6n^6}, \\
L_3 & = & - \frac{a^2}{6 n (n^2 - a^2)} - \frac{1}{180 n^3} , \\
L_4 & = & \frac{a^2}{n} .
\end{eqnarray*}
Putting these pieces together and rearranging we find that
\begin{align}
r_n (a) \ge &\left[\frac{31 a^2}{64 n^2}+\frac{a^4}{4n^4}+\frac{a^6}{6n^6}
                      -\frac{a^4}{6n^3}-\frac{a^6}{15n^5}-\frac{a^8}{28n^5}\left( \frac{1}{n^2-a^2}\right)\right]\nonumber \\
&+\left[\frac{a^2}{64n^2} +\frac{1}{6 n} - \frac{1}{180 n^3} -\frac{1}{12(n+a)}-\frac{1}{12(n-a)}\right] \label{second_exp}\\
&=: K_1+K_2> 0 \label{third_exp}
\end{align}
will prove the claim. Note in \eqref{second_exp} that the $a^2/n$ term 
cancelled by virtue of the lower bound estimate based on the Taylor
expansion of $(1+x)\log(1+x)+(1-x)\log(1-x)$. First note that
\begin{eqnarray*}
K_2 
& = & 
\frac{a^2}{64n^2}+\frac{1}{6 n} - \frac{1}{180 n^3} -\frac{1}{12(n+a)}-\frac{1}{12(n-a)}\\
& = &
\frac{a^2[28n^3-45 a^2 n]+a^2[16 n^3-480n^2]+[a^2 n^3+16 a^2-16 n^2]}{2880 n^3 (n-a) (n+a)}
\end{eqnarray*}
The denominator of the right-hand-side is clearly positive for 
$a\in\left\{1,2,\dots,\lfloor\sqrt{2n}\rfloor\right\}$. By inspection, we can see the term 
$a^2 n^3+16a^2-16n^2$  
in the numerator is increasing in $a$. Picking $a=1$, we then see 
$n^3+16-16n^2 > 0$ for $n \geq 31$, and thus $a^2 n^3+16a^2-16n^2>0$ for all admissible
$a$.  
Next, the polynomial $28 n^3-45 a^2 n$ is decreasing in the admissible $a$. 
For any fixed $n$, the minimum value it can attain is then larger than
$28 n^3- 90n^2$. For $n\geq 31$, this quantity is positive. 
Therefore,  $28 n^3-45 a^2 n > 0$ for all admissible $a$ when $n \geq 31$. Finally, 
note that $16n^3-480n^2 = 16 n^2 (n - 30) > 0$ for $n\ge 31$.
Hence we have shown
$K_2 > 0$.

We next have
\begin{eqnarray}
K_1 
& = & \left[\frac{31a^2}{64n^2}-\frac{a^4}{6n^3}\right] 
        +  \left[\frac{a^4}{4n^4}-\frac{a^6}{15n^5}\right]
        +   \left[\frac{a^6}{6n^6}-\frac{a^8}{28n^5}\left( \frac{1}{n^2-a^2}\right)\right]  \nonumber \\
&= & \left[\left(\frac{a^2}{192 n^3}\right)(93  n-32 a^2)\right] 
        + \left[\left(\frac{a^4}{60 n^5}\right)(15n-4 a^2)\right] \nonumber\\
       && \qquad  \ \ \  + \ 
       \left[ \left(\frac{a^6}{84 n^6 \left(n^2-a^2\right)}\right)(14  n^2-3 a^2 n-14 a^2)
        \right]
\nonumber \\
&\equiv &  \left[\left(\alpha\right)(93  n-32 a^2)\right]
                   + \left[\left(\beta\right)(15n-4 a^2)\right] \nonumber \\
&& \ \ \ + \  \left[   \left(\gamma \right)(14  n^2-3 a^2 n-14 a^2)\right]\ .
\label{I1_EXP}
\end{eqnarray}
Again since $a \in \left\{1,\dots,\lfloor\sqrt{2n}\rfloor\right\}$, it is clear that $\alpha,\beta,$ and 
$\gamma$ in \eqref{I1_EXP} are positive for all admissible choices of $a$.
Hence, the sign of each bracketed term will be dictated by the remaining polynomial in $a$. 
It is also clear from their form that each polynomial is decreasing in $a$; hence we need
only evaluate at the endpoints to determine positivity. But 
$93n-32(\sqrt{2n})^2 = 29n>0$, $15n-4(\sqrt{2n})^2 = 15n-8n=7n>0$, 
and $14  n^2-3 (\sqrt{2n})^2 n-14 (\sqrt{2n})^2=14n^2-6n^2-28n=4n(2n-7)>0$
with the final inequality following as $n \geq 31$. Hence all terms in \eqref{I1_EXP} are positive and so $K_1 > 0$. 
Together with  $K_2 > 0$ as proved above, the claim is proved for $n\ge 31$.

Since the bound holds for $a \in \{ 1, \ldots , \lfloor \sqrt{2n} \rfloor \}$ and $n \in \{ 1, \ldots , 30 \}$ by 
direct numerical computation, the claim follows. 
\end{proof}
 
\section{Some comparisons and connections}

\subsection{Comparisons:  two-sided tail bounds}
Here we compare and contrast our results with those of \cite{MR2887482}.  As in 
\cite{MR2887482} (see also \cite{Wei-Dudley:11}), we say that {\sl the DKW inequality holds for given $m,n$ and $C$} if 
\begin{eqnarray*}
P( D_{m,n} \ge t ) \le C \exp ( - 2 t^2 )  \ \ \ \mbox{for all} \ \ t>0,
\end{eqnarray*}
and we say that {\sl the DKWM inequality holds for given $m,n$} if the inequality in the 
last display holds with $C=2$.  \cite{MR2887482} prove the following theorem:
\smallskip

\begin{thm} 
\label{WeiDudleyMainThm} 
(Wei and Dudley, 2012)\ \ 
For $m=n$ in the two sample case:\\
(a) The DKW inequality always holds with $C = e \dot= 2.71828 $.\\
(b) For $m=n\ge 4$, the smallest $n$ such that $H_c$ can be rejected at level $0.05$, 
the DKW inequality holds with $C=2.16863$.  \\
(c) The DKWM inequality holds for all $m=n \ge 458$. \\
(d) For each $m=n < 458$, the DKWM inequality fails for some $t$ of the form $t=  k/\sqrt{2n}$.\\
(e) For each $m=n< 458$, the DKW inequality holds for $C=2(1+\delta_n)$ for some $\delta_n>0$ 
where, for $12\le n \le 457$, 
$$
\delta_n < - \frac{0.07}{n} + \frac{40}{n^2} - \frac{400}{n^3} .
$$
\end{thm}
\smallskip

For comparison, the following theorem follows from Theorem~\ref{BasicGW-TwoSampleKS-bounds}. 
We say that {\sl the modified DKWM inequality holds for given $m,n$} if 
\begin{eqnarray*}
P( D_{m,n} \ge t ) \le  2\exp \left ( - 2 \left ( \frac{N-1}{N} \right ) t^2 \right )  \ \ \ \mbox{for all} \ \ t>0,
\end{eqnarray*}
\smallskip

\begin{thm} 
\label{thm:TwoSampleKS-ourBoundSummary}
For $m=n$ in the two sample case:\\
(a)  For all $n\ge1$ the modified DKWM inequality holds.  \\
(b)  Alternatively,  for the modified Kolmogorov statistic given by 
$$
D_{m,n}^{mod} \equiv \sqrt{\frac{N-1}{N}} \sqrt{\frac{mn}{N}} \| \FF_m - \GG_n \|_{\infty} ,
$$
the DKWM inequality holds for all $n\ge1$.
\end{thm}
\smallskip

We are not claiming that our ``modified'' version of the DKWM inequality improves on the results of 
\cite{MR2887482}:  it is clearly worse for $m=n> 458$.    On the other hand, it may provide a useful
clue to the formulation of DKWM type exponential bounds for two-sample Kolmogorov statistics 
when $m \not= n$.  In this direction we have the following conjecture:
\smallskip

\par\noindent
{\bf Conjecture:}   For any $m\not= n$, 
\begin{eqnarray}
&& P \left ( D_{m,n}^+ > t \right ) \le \exp \left ( - 2 \left ( \frac{N-1}{N} \right ) t^2 \right ) \ \ \mbox{for all} \ \ t>0 
\label{OneSidedModifiedDKWMGeneralMandN}\\
&& P \left ( D_{m,n} > t \right ) \le 2 \exp \left ( - 2 \left ( \frac{N-1}{N} \right ) t^2 \right ) \ \ \mbox{for all} \ \ t>0 .
\label{TwoSidedModifiedDKWMGeneralMandN}
\end{eqnarray}

That is, we conjecture that the modified DKWM inequality holds for all $m,n \ge1$.  
This is supported by all the numerical experiments we have conducted so far.  

\subsection{Comparisons:  one-sided tail bounds}

\cite{MR2887482}    
do not treat bounds for the one-sided statistics.  
Here we summarize our results with a theorem which parallels their 
Theorem~\ref{WeiDudleyMainThm} above.
In analogy with their terminology, we say that
{\sl the one-sided DKW inequality holds for given $m,n$ and $C$} if 
\begin{eqnarray*}
P( D_{m,n}^+ \ge t ) \le C \exp ( - 2 t^2 )  \ \ \ \mbox{for all} \ \ t>0,
\end{eqnarray*}
and we say that {\sl the one-sided DKWM inequality holds for given $m,n$} if the inequality in the 
last display holds with $C=1$.     
Moreover, we say that {\sl the modified one-sided DKWM inequality holds for given $m,n$} if 
\begin{eqnarray*}
P( D_{m,n}^+ \ge t ) \le  \exp \left ( - 2 \left ( \frac{N-1}{N} \right ) t^2 \right )  \ \ \ \mbox{for all} \ \ t>0.
\end{eqnarray*}

\smallskip
\begin{thm}
\label{thm:OneSideTwoSampleKS-bound}
For $m=n$ in the two sample case:\\
(a) The one-sided DKW inequality holds for all $n\ge 1$ with $C = e/2 \dot= 2.71828 /2 $ \\ 
$\phantom{blab}= 1.35914 $.  For this range of $n$, $C=e/2$ is sharp since equality occurs \\
$\phantom{blab}$for $n=1$ and
$t=1/\sqrt{2}$ (or $a = t \sqrt{2n} = 1$).\\
(b) For $m=n\ge 5$, 
the one-sided DKW inequality
holds with $C=2.16863/2=$\\
$\phantom{blab}1.084315$.  \\
(c) The one-sided DKWM inequality fails for all $m=n \ge 1$. \\
(d) The modified one-sided DKWM inequality holds for all $m=n\ge1$.
\end{thm}

\begin{proof}[\bf Proof]
(c) follows from Theorem~\ref{BasicGW-TwoSampleKS-bounds}-B.  
(d) follows from Theorem~\ref{BasicGW-TwoSampleKS-bounds}-A.
It remains only to prove (a) and (b).

To prove (a), we first note that \cite{MR2887482} showed that for $n\ge108$ we have
\begin{eqnarray*}
\frac{{2n \choose n+a}}{ {2n \choose n}} 
& < & \exp ( - a^2/n) \ \ \ \mbox{for} \ \ \sqrt{3n} \le a \le n \\
& < & (e/2) \exp ( - a^2/n) .
\end{eqnarray*}
Thus to prove that the claimed inequality holds for $n\ge 108$, it suffices to show that 
it holds for $ t_0 \sqrt{n} \le a \le \sqrt{3} \sqrt{n}$ where $t_0 \equiv \sqrt{(1/2)\log(e/2)}$ is 
the smallest value of $t$ for which the bound is less than or equal to $1$.

Proceeding as in the proof of Theorem~\ref{BasicGW-TwoSampleKS-bounds}-A, we find that we want to show that
\begin{eqnarray*}
\log \frac{n! n!}{(n+a)! (n-a)!} + \frac{a^2}{n} - \log (e/2)  < 0 \ \ \mbox{for} \ \ t_0 \sqrt{n}\le a \le \sqrt{3} \sqrt{n} .
\end{eqnarray*}
By the same arguments used in the proof of Theorem~\ref{BasicGW-TwoSampleKS-bounds}-A, 
we find that the left side in the last display is bounded above by 
\begin{eqnarray*}
&& - \frac{a^4}{6n^3} - \frac{a^6}{15 n^5} - \frac{a^8}{28n^7} + \frac{a^4}{4 n^4} + \frac{a^6}{6 n^6 (1 - a^2/n^2)} + I_3\\
&& \qquad + \ \frac{a^2}{2n^2} - \log (e/2) \\
&& \ \ \equiv  K_1 + K_2 .
\end{eqnarray*}
Now $K_1 \le 0$ for $n \ge 4$ and $a \in \{ 1, \ldots , n-1 \}$ by the previous proof, and 
\begin{eqnarray*}
K_2 \equiv \frac{a^2}{2 n^2} - \log (e/2) < 0 \ \ \mbox{for all} \ \ a \le \sqrt{3}\sqrt{n} 
\end{eqnarray*}
if
$$
\frac{3}{2n} < \log (e/2) , \ \ \ \mbox{or} \ \ \ n > \frac{3}{2 \log (e/2)} \dot= 4.888\ldots .
$$
This completes the proof for $n \ge 108$.  Numerical computation easily shows that the claim holds 
for all $n \in \{ 1, \ldots , 107 \}$.

The proof of (b) is similar upon replacing $e/2$ by $1.084315$, 
and again computing numerically for $n \in \{ 1, \ldots , 107 \}$. 
\end{proof}

\begin{cor} 
For $n\ge 5$ and $C = 1.084315$,
\begin{eqnarray*}
P(D_{n,n}^+ \ge t) 
& \le & \min \left \{ \exp \left ( -2 \left (1 - 1/N \right ) t^2 \right ) , \ C \exp (- 2 t^2 ) \right \} \\
& = & \left \{ \begin{array}{l l} C \exp(- 2t^2), & t \ge t_0 \equiv \sqrt{n \log C} \dot= .285 \sqrt{n} , \\
                                               \exp ( - 2 (1-1/N)t^2) , & t \le t_0 \equiv \sqrt{n \log C} .
                    \end{array} \right .
\end{eqnarray*}
\end{cor}

Figures~\ref{fig1} and~\ref{fig2} illustrate Theorem~\ref{thm:OneSideTwoSampleKS-bound}.

\begin{figure}[h]
  \begin{center}
    \includegraphics[keepaspectratio=true,scale=0.8]{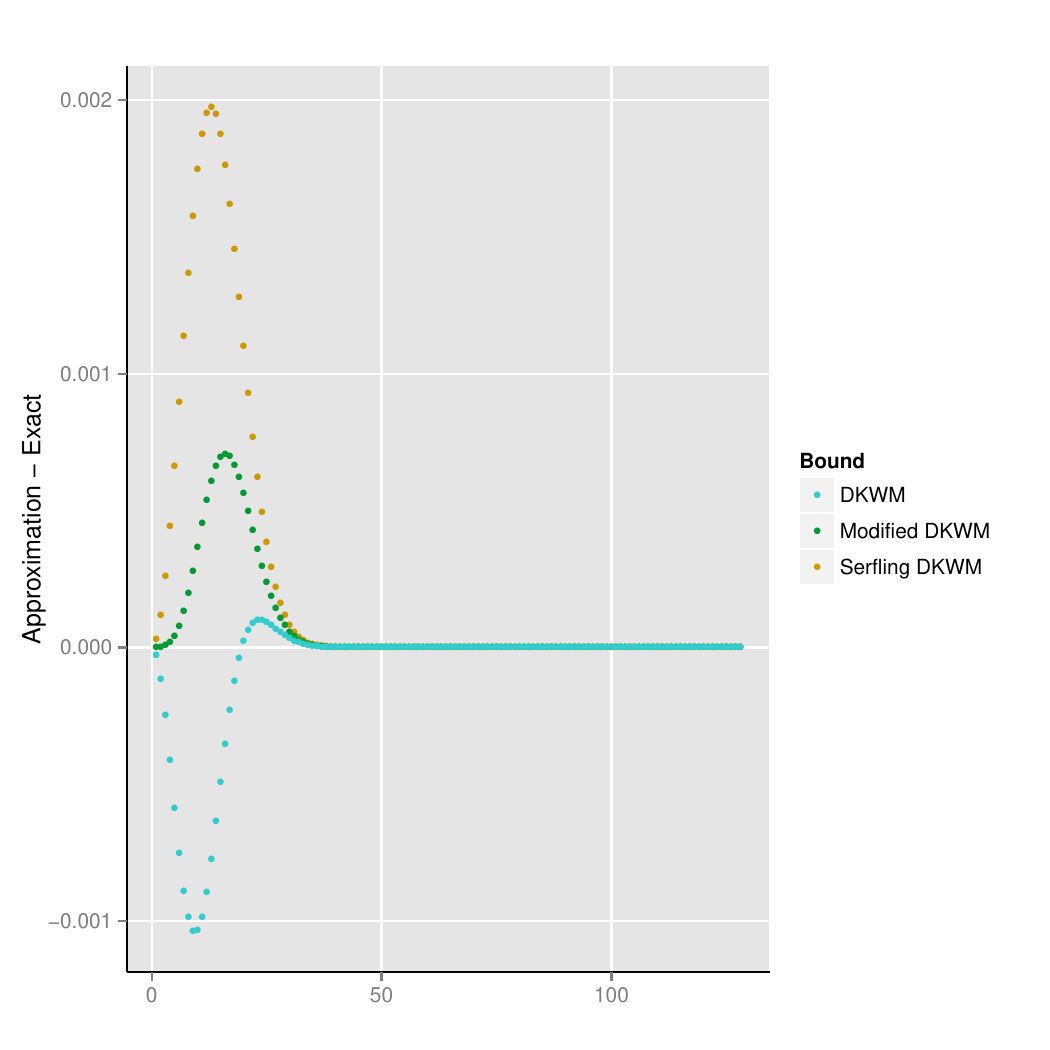}
  \end{center}
  \caption{Difference between approximations and exact one-sided probabilities $P\left(D_{n,n}^+ > t\right)$ 
    for $n=128$ and $a \in \left\{1,2,\dots,128\right\}$. Negative values indicate the exact probability exceeds the approximation.
           Serfling DKWM is the bound obtained via the heuristic of section 2, using the sampling fraction $1-f_n^* = (N-n+1)/N$. Modified
           DKWM uses the sampling fraction $1-f_n=(N-n)/(N-1)$. DKWM uses the fraction from Wei and Dudley.
  \label{fig1}}
\end{figure}

\begin{figure}[h]
  \begin{center}
    \includegraphics[keepaspectratio=true,scale=0.8]{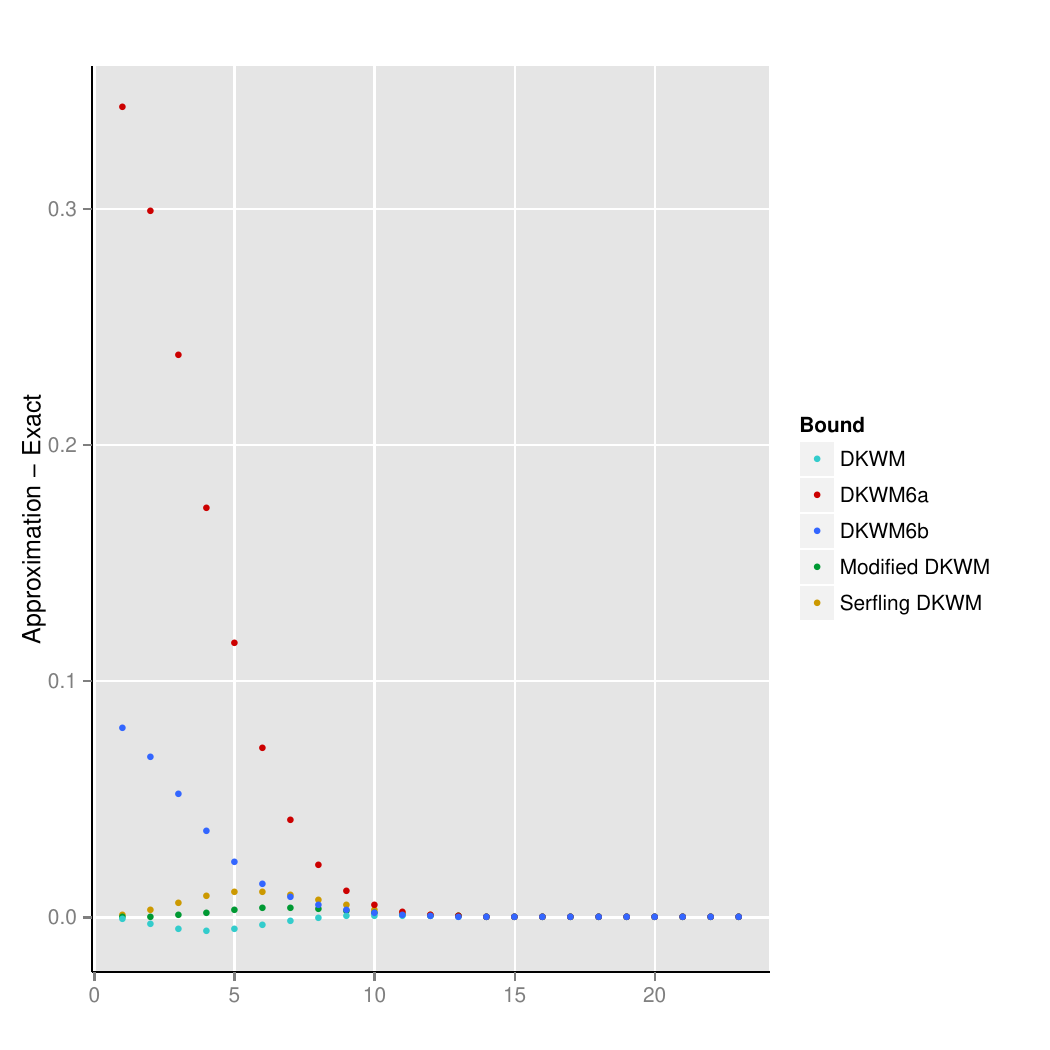}
  \end{center}
  \caption{Difference between approximations and exact one-sided probabilities $P\left(D_{n,n}^+ > t\right)$ 
    for $n=23$ and $a \in \left\{1,2,\dots,23\right\}$. Negative values indicate the exact probability exceeds the approximation.
    DKWM6a corresponds to the DKWM bound with the constant $e/2$, discussed in Theorem 6(a). 
    DKWM6b corresponds to the DKWM bound with the constant $2.16863/2$,
    discussed in Theorem 6(b).
   \label{fig2}}
\end{figure}

\section*{Acknowledgements}  
The second author owes thanks to Werner Ehm for several helpful conversations 
 and to Martin Wells for pointing out the Pitman reference.
 We also owe thanks to the referee for a number of helpful comments and suggestions.
 The improved inequality (and proof) in Theorem 3 part A is due to Lutz D\"umbgen.

\bigskip

\end{document}